\documentclass[a4paper,11pt]{amsart}

\author{Giles Gardam}
\title{Non-trivial units of complex group rings}

\usepackage[T1]{fontenc}
\usepackage[utf8]{inputenc}
\usepackage{amsfonts}
\usepackage{amsmath}
\usepackage{amssymb}
\usepackage{amsthm}
\usepackage{mathpazo}
\usepackage{mathtools}

\usepackage{hyperref}

\theoremstyle{plain}
\newtheorem{thm}{Theorem}
\newtheorem*{cor}{Corollary}

\theoremstyle{definition}
\newtheorem{rmk}{Remark}

\newcommand{\C}{\mathbb{C}}
\newcommand{\F}{\mathbb{F}}
\newcommand{\R}{\mathbb{R}}
\newcommand{\Q}{\mathbb{Q}}
\newcommand{\Z}{\mathbb{Z}}
\newcommand{\gp}[2]{\langle \, #1 \, | \, #2 \, \rangle}
\newcommand{\gen}[1]{\langle #1 \rangle}

\DeclareMathOperator{\supp}{supp}

\allowdisplaybreaks

\address{Mathematisches Institut, Universität Bonn, Endenicher Allee 60, 53115 Bonn, Germany}
\email{gardam@math.uni-bonn.de}

\begin{document}

\begin{abstract}
    The Kaplansky unit conjecture for group rings is false in characteristic~zero.
\end{abstract}

\maketitle

\section{Introduction}

Let $G$ be a torsion-free group and $K$ be a field.
The question of whether the group ring $K[G]$ can have any units other than the \emph{trivial units}, i.e.\ the non-zero scalar multiples of group elements, dates back to Higman's thesis \cite[p.~77]{Higman40} and is generally known as the Kaplansky unit conjecture.
An important consequence of a given $K[G]$ satisfying the conjecture is that it has no zero divisors \cite[Lemma 13.1.2]{Passman85}.

Once a counterexample to the unit conjecture was given in characteristic~$2$ \cite{Gardam21} and then generalized to arbitrary positive characteristic \cite{Murray21}, the natural question was whether this phenomenon exists in characteristic~$0$ or is simply an accident of positive characteristic; the dependence on the Frobenius endomorphism in \cite{Murray21} strongly hinted at the latter.
The characteristic~$0$ case is however the most interesting: the topological motivation such as underlies Higman's thesis is focussed on the integral group ring $\Z[G]$ and in an analytic setting such as operator algebras one generally restricts attention to $K = \C$.
For instance, one way to give a counterexample to the Atiyah conjecture on integrality of $L^2$-Betti numbers would be to find $G$ such that $\C[G]$ has zero divisors, but this necessitates $\C[G]$ having non-trivial units.
Moreover, a counterexample in characteristic~$0$ necessarily gives a counterexample in characteristic~$p$ for all but finitely many~$p$.

\section{The counterexample}

\begin{thm}
    \label{thm:nontrivial_unit}
    Let $P$ be the torsion-free group defined by the presentation $\gp{a, b}{b^{-1} a^2 b = a^{-2}, \, a^{-1} b^2 a = b^{-2}}$.
    Then $\C[P]$ has non-trivial units.
    For example, set $x = a^2, y = b^2, z = (ab)^2$, let $\zeta_8$ be a primitive 8th root of unity and let $i = \zeta_8^2$.
    Then
    \begin{align*}
        1 &+ i \big(x - x^{-1} - y + y^{-1} \big) z^{-1} \\
        &+ \zeta_8 \big(-i x^{-1} + 1 - x^{-1} y^{-1} z + i y z \big) a \\
        &+ \zeta_8 \big( x^{-1} y^{-1} z - i x z + i y^{-1} z^2 - z^2 \big) b \\
        &+ i \big( -i x^{-1} - i y - 1 + x^{-1} y \big. \\
          &\phantom{+i\big(}+ \big. i x^{-1} z^{-1} + i y z^{-1} - z^{-1} + x^{-1} y z^{-1} \big) z^{-1} ab
    \end{align*}
    is a non-trivial unit in $\C[P]$.
\end{thm}

Surprisingly, the exact same $21$-element subset of the group that supports the non-trivial unit over $\F_2$ given in \cite{Gardam21} also supports a non-trivial unit over $\C$.
(For aesthetic reasons we have taken the original support and multiplied it on the right by $(ab)^{-1}$ and then applied the automorphism $a \mapsto a, b \mapsto a^{-2} b$ to arrive at the support of Theorem~\ref{thm:nontrivial_unit}.)
We have not succeeded in finding a non-trivial unit in $\Z[P]$ but note that the coefficients in the theorem are at least algebraic integers.

\begin{proof}[Proof of Theorem~\ref{thm:nontrivial_unit}]
    This is readily verified using computer algebra.
    There is actually a $2$-parameter family of solutions in terms of primitive 8th roots of unity.
    Let $R = \Z[s, t] / \langle s^4 + 1, t^4 + 1 \rangle$.
    Then setting
    \begin{align*}
        \alpha_1 &= 1 + t^2 x z^{-1} - t^2 x^{-1} z^{-1} - s^2 y z^{-1} + s^2 y^{-1} z^{-1} \\
        \alpha_a &= -s^2 x^{-1} + 1 - x^{-1} y^{-1} z + s^2 y z \\
        \alpha_b &= x^{-1} y^{-1} z - t^2 x z + t^2 y^{-1} z^2 - z^2 \\
        \alpha_{ab} &= -t^2 x^{-1} z^{-1}  -s^2 y z^{-1} + s^2 t^2 z^{-1} + x^{-1} y z^{-1} \\
        &\phantom{=} + s^2 x^{-1} z^{-2} + t^2 y z^{-2} - z^{-2} - s^2 t^2 x^{-1} y z^{-2}
    \end{align*}
    gives a unit $\alpha_1 + s \alpha_a a + t \alpha_b b + s t \alpha_{ab} ab$ in $R[P]$ whose inverse $\beta$ is defined analogously in terms of
    \begin{align*}
        \beta_1 &= 1 + t^2 x^{-1} z - t^2 x z + s^2 y z -s^2 y^{-1} z\\
        \beta_a &= -1 + s^2 x^{-1} -s^2 y z + x^{-1} y^{-1} z \\
        \beta_b &= -x^{-1} y^{-1} z + t^2 x z -t^2 y^{-1} z^2 + z^2 \\
        \beta_{ab} &= s^2 t^2 x^{-1} y + 1 - t^2 y - s^2 x^{-1} \\
        &\phantom{=} -x^{-1} y z -s^2 t^2 z + s^2 y z + t^2 x^{-1} z.
    \end{align*}
    Sample \texttt{sage} code verifying this is available at the repository of accompanying code to this paper \cite{gardam_2024_14008425} (and included as an ancillary file to the arXiv version of this paper).
    Specializing $s = t = \zeta_8$ gives the unit of the theorem statement.
    The group $P$ arises as a group of affine isometries of Euclidean space $\R^3$ and thus can be conveniently implemented using a faithful representation, namely \[
        a \mapsto \begin{pmatrix}
            1 & 0 & 0 & 1 \\
            0 & -1 & 0 & 1 \\
            0 & 0 & -1 & 0 \\
            0 & 0 & 0 & 1 \\
        \end{pmatrix},
        \quad
        b \mapsto \begin{pmatrix}
            -1 & 0 & 0 & 0 \\
            0 & 1 & 0 & 1 \\
            0 & 0 & -1 & 1 \\
            0 & 0 & 0 & 1 \\
        \end{pmatrix}.
    \] Indeed, as computed in \cite{Gardam21} by hand from the group presentation, the group $P$ is torsion-free and the index-4 subgroup $\gen{x, y, z}$ is isomorphic to $\Z^3$ and faithfulness on this subgroup is seen immediately.
\end{proof}

Since the unit has coefficients in $\Z[\zeta_8]$, it yields units in characteristic~$p$ for \emph{all} primes $p$ and not simply for all but finitely many $p$.
To be precise:

\begin{cor}
    The $21$-element set of the theorem supports non-trivial units over $\F_{p^2}$ for any prime $p$, or $\F_p$ if $p = 2$ or $p \equiv 1 \mod 8$.
\end{cor}

\begin{proof}
    The only requirement on the field is that there be a root of $t^4 + 1$.
    The corollary follows as $\F_{p^k}^\times \cong \Z / (p^k - 1) \Z$.
\end{proof}

We now make a few related remarks.
These are phrased in the generality of the unit $\alpha \in R[P]$ constructed in the proof of Theorem~\ref{thm:nontrivial_unit} rather than just one of its images in $\C[P]$.

\begin{rmk}
    \label{remark:symmetry}
The unit $\alpha$ is \emph{symmetric} and \emph{twisted unitary} in the following sense, as noted by Bartholdi for positive characteristic units \cite{Bartholdi23}.
Let $\phi_0 \colon a \mapsto a^{-1}, b \mapsto b^{-1}$ and $\phi_1 \colon a \mapsto a, b \mapsto b^{-1}$ be automorphisms of $P$ and let $\chi_0 \colon a \mapsto -s^2, b \mapsto -t^2$ and $\chi_1 \colon a \mapsto s^2, b \mapsto -1$ be homomorphisms $P \to R^\times$.
A group automorphism extends to a group ring automorphism and a character $\chi \colon P \to R^\times$ induces a \emph{gauge automorphism} of $R[P]$ that extends $g \mapsto \chi(g) g$.
Putting them together, the automorphisms of $R[P]$ defined by \[
    \theta_0(\sum_g \lambda_g g) = \sum_g \chi_0(g) \lambda_g \phi_0(g)
\] and \[
    \theta_1(\sum_g \lambda_g g) = \sum_g \chi_1(g) \lambda_g \phi_1(g)
\]
satisfy $\theta_0(\alpha) = \alpha$ and $\theta_1(\alpha)^* = \alpha^{-1}$, that is, $\alpha$ is $\theta_0$-symmetric and $\theta_1$-unitary.
Here we write $^* \colon R[P] \to R[P]$ for the anti-involution extending $g \mapsto g^{-1}$.
\end{rmk}

\begin{rmk}
    \label{remark:almost_integral}
    There is a homomorphism $\rho \colon P \to R^\times / \{\pm 1\}$ defined by mapping $a \mapsto \{\pm s\}, b \mapsto \{\pm t\}$, which has image $\gen{s, t} / \{\pm 1\} \cong \Z/4 \oplus \Z/4$.
    The unit $\alpha \in R[P]$ has the property that it is of the form $\sum \lambda_g g$ where each non-zero coefficient $\lambda_g \in \rho(g)$.
    This property can be thought of as a type of grading.
    It can also be expressed in terms of invariance under the group ring automorphism that applies the ``complex conjugation'' automorphism $R \to R \colon s \mapsto s^{-1}, t \mapsto t^{-1}$ followed by the gauge automorphism corresponding to the character $a \mapsto s^2, b \mapsto t^2$; this allows us to rephrase the symmetry expressed in Remark~\ref{remark:symmetry} in terms of complex conjugation.
    As the abelianization of $P$ is $\Z/4\Z \oplus \Z/4\Z$, we cannot lift $\rho$ to a homomorphism $P \to \Z/8\Z \oplus \Z/4\Z \cong \gen{s, t} \leq R^\times$, which would otherwise allow us to ``untwist'' the unit into an element of $\Z[P]$ via the would-be gauge automorphism corresponding to the would-be character $a \mapsto s^{-1}, b \mapsto t^{-1}$.
\end{rmk}

\begin{rmk}
    The unit $\alpha \in R[P]$ has image $1 \in R[\Z/4 \oplus \Z/4]$ under the ring homomorphism induced by abelianization of $P$.
    This image need not be trivial \emph{a priori}: while $\Z[\Z/4 \oplus \Z/4]$ only has trivial units, $\Z[\zeta_8][\Z/4 \oplus \Z/4]$ has non-trivial units \cite{Higman40}.
\end{rmk}

\section{Finding the solution}

The problem of finding a non-trivial unit in $\C[P]$ resisted many attempts of the author over a period of three years and surely attracted the attention of many others; this problem of course looks easier in hindsight.
A theoretical approach to the problem was pursued in \cite{ClaramuntGrabowski23}, where a criterion for the existence of non-trivial units in $K[P]$ was elaborated.

A very natural idea is to attempt to lift the solution over $\Z/2\Z$ to $\Z/2^n\Z$ for increasing $n$ so as to arrive at a solution over the ring of $2$-adic integers.
The simple obstacle here is that $\Z_2$ has no square root of $-1$ whereas all non-trivial characteristic~$0$ units supported on those $21$-element sets require an 8th root of unity as explained below (there could however be other units over $\Z/2\Z$ for which the $2$-adic approach works).

Finding a unit such that it and its inverse are supported on the corresponding $21$-element sets means solving a large system of quadratic equations in $42$ variables.
Code generating and working with this system is available for the curious reader at the zenodo repository \cite{gardam_2024_14008425}.
Let $g_1, \dots g_{21} \in P$ be the elements of the support of $\alpha$, enumerated in the same order as in Theorem~\ref{thm:nontrivial_unit}, and let $h_1, \dots, h_{21}$ be the support of its inverse $\beta$ as given in the proof.
We wish to solve for $\alpha \beta = 1$ where our variable group ring elements are $\alpha = \sum_{i=1}^{21} u_i g_i$ and $\beta = \sum_{i=1}^{21} v_i h_i$, given in terms of the variables $u_i$ and $v_i$.
This defines a system of $121$ quadratic equations in $\Z[u_1, \dots, u_{21}, v_1, \dots v_{21}]$.
For instance, there are $17$ pairs with $g_i h_j = 1$ and accordingly the equation corresponding to the coefficient of $\alpha \beta$ at the identity is \begin{multline*}
    u_1 v_1 + u_2 v_2 + u_3 v_3 + u_4 v_5 + u_5 v_4 + u_6 v_6 + u_7 v_7 + u_{12} v_{13} + u_{13} v_{12} \\
    + u_{14} v_{17} + u_{15} v_{16} + u_{16} v_{15} + u_{17} v_{14} + u_{18} v_{21} + u_{19} v_{20} + u_{20} v_{19} + u_{21} v_{18} = 1.
\end{multline*}
The other $120$ equations are homogeneous (in fact bi-linear), such as \[
    u_1 v_2 + u_{12} v_{11} + u_{14} v_{19} + u_{17} v_{20} = 0, \quad
 u_1 v_3 + u_{13} v_{10} + u_{15} v_{18} + u_{16} v_{21} = 0
\] and each is the sum of an even number of monomials (as the two $21$-element sets define units in $\F_2[P]$).
At this point any solution can be modified by a scalar $\lambda \in \C^\times$, replacing $u_i$ by $\lambda u_i$ and $v_i$ by $\frac{1}{\lambda} v_i$, which we would like to factor out.
Thus one should assume for example that the units are \emph{normalized} i.e.\ add the equations \begin{equation}
    \label{equation:normalization}
    \tag{$*$}
    \sum_{i=1}^{21} u_i = 1, \quad \sum_{i=1}^{21} v_i = 1.
\end{equation}
This means in particular that the trivial units are a $0$-dimensional set comprising $17$ points.
We have chosen a convenient enumeration of the $42$ elements such that the symmetry at the group level expressed in Remark~\ref{remark:symmetry} manifests itself in the system of equations in the following way: the set of equations is invariant both under swapping $u_i \leftrightarrow v_i$, and under fixing $u_1, v_1$ while swapping \[
    u_2 \leftrightarrow u_3, \, u_4 \leftrightarrow u_5, \, \dots, \, u_{20} \leftrightarrow u_{21}, \, v_2 \leftrightarrow v_3, \, \dots, \, v_{20} \leftrightarrow v_{21}.
\]

After Bartholdi's coherent reformulation of \cite[Lemma 1]{Gardam21} in terms of automorphisms of the group ring \cite{Bartholdi23}, one could attempt to solve the system of quadratic equations over $\C$ by adding additional constraints relating variables with each other according to the automorphisms of $\C[P]$.
As $P$ has abelianization $\Z/4 \oplus \Z/4$ we have the freedom to consider characters taking values in $\{\pm 1, \pm i\}$ and not just $\{\pm 1\}$ as Bartholdi did.
That reduces the number of variables from $42$ to $11$ ($\phi_0$ has precisely one fixpoint in $\supp(\alpha)$, namely $1$), for instance $u_1, u_2, u_4, u_6, \dots, u_{20}$.
It seems to be more efficient to enumerate over the $4^4$ choices of a pair of characters $\chi_0, \chi_1$ than to express them using additional variables (even if some choices do not define an anti-involution $\alpha \mapsto \theta_1 (\alpha)^*$).
The resulting collection of systems of equations can be solved in a matter of seconds for example using \texttt{singular} \cite{DGPS} via \texttt{sage} \cite{sagemath}, even when performing the Gröbner basis computation directly over $\Q$ instead of over $\F_p$ for some large prime $p$.

The automorphisms $\theta_0$, $\theta_1$ of $\C[P]$ are arguably unnatural, as one does not get the desirable property of pairs of elements $\alpha$, $\beta$ satisfying the symmetry as described by \cite[Lemma 1]{Gardam21} that $\alpha \beta$ automatically vanishes outside an index $2$ subgroup of $P$.
Nonetheless, such a trick can only work for virtually abelian groups, whereas one naturally wishes to understand the units of other torsion-free groups.
The author knows one other torsion-free group supporting non-trivial units over $\F_2$, and here we again have symmetry but in an unexpected way, emphasizing the point that \cite[Lemma 1]{Gardam21} is not the end of the story of symmetry for units.
This is presented below in Section~\ref{section:beyond}.

However, it turns out that one can solve the problem without imposing these symmetry constraints, in a more ``brute force'' fashion, using the state of the art software \texttt{msolve} \cite{Berthomieu21}.
The time needed to compute a Gröbner basis is only on the scale of hours but one needs a machine with generous memory\footnote{using \texttt{msolve v0.7.1} with a single thread on a 2.1 GHz Intel processor took $4.5$ hours and $19$ GB of RAM}.
This Gröbner basis itself has limited value, as the computation is performed modulo a large prime $p$ (we fixed the ``random'' prime $1000000007$ for reproducibility), the variety it defines has dimension $0$, and the 3490 basis polynomials are extremely complicated (comprising over $1.8$ million terms with over $0.8$ million different coefficients in $\F_p$!).
As the system has $17$ isolated trivial solutions, this is perhaps not surprising.

We can however avoid the issue of the trivial solutions by ``localizing'' a pair of the variables, that is, introducing new variables as their multiplicative inverses.
This divides the problem of finding a non-trivial unit into $\binom{21}{2}$ cases for the smallest indices $i < j$ such that $u_i, u_j \neq 0$.
For simplicity, we can replace the normalization equations \eqref{equation:normalization} with $u_i = 1$ and then we only need one additional variable $w$ such that $u_j w - 1 = 0$.
\emph{A priori} there could be solutions among the $\binom{21}{2} - 1$ cases where either $u_1 = 0$ or $u_2 = 0$, but these are quickly ruled out either by determining that the Gröbner basis is the trivial basis $[1]$ in each case or (without committing to any characteristic) verifying with a SAT or SMT solver such as \texttt{z3} \cite{DeMouraBjorner2008} that no proper subsets of the two $21$-element sets falsify the \emph{two unique product property} i.e.\ for candidate proper subsets to be the support, the resulting system always contains an equation $u_k v_l = 0$, contradicting the assumption that $g_k$ and $h_l$ are in the support.
Everything thus comes down to the case where $u_1, u_2 \neq 0$.
This system of equations is much easier to solve: we can compute a Gröbner basis with \texttt{msolve} in under 20 seconds (or under 50 minutes with \texttt{singular}).
After doing this, the computed Gröbner basis being non-trivial already tells us that there is a non-trivial solution, at least in large characteristic $p$.
Even better: in a minor miracle, the coefficients are $\pm 1 \in \F_p$ so that it is clear how to lift to a Gröbner basis over $\Q$ and thus extract solutions over $\C$ (of which there are exactly $16$ as can be immediately computed from the Gröbner basis).
By inspection, one quickly realizes that the variables all take values that are $8$th roots of unity and in fact all the solutions can be parametrized in terms of $2$ primitive $8$th roots, as done in the proof of Theorem~\ref{thm:nontrivial_unit}, specializing to $4^2$ different complex solutions.
Thus, modulo the unlikely possibility that $p = 1000000007$ is a bad prime for this system of equations, these $16$ solutions are all the non-trivial solutions over $\C$.

\section{Beyond virtually abelian groups}
\label{section:beyond}

Let \[
    S = \gp{x, y}{(xy)^2 (xy^{-1})^2, (yx)^2 (yx^{-1})^2}
\] be the virtually nilpotent non-unique product group identified in \cite[p.~23]{Soelberg18} (see also \cite{NielsenSoelberg24}), where it is presented as $\gp{a, b}{a^{-1} b^2 a b^2, a^{-2} b a^{-2} b^3}$; this is isomorphic to $S$ via $x \mapsto a, y \mapsto a b^{-1}$.
It has a faithful representation \[
    x \mapsto \begin{pmatrix*}[r]
        -1 &  1 &  0 \\
         0 & -1 &  0 \\
         0 &  0 &  \phantom{-}1
    \end{pmatrix*},
    \quad
    y \mapsto \begin{pmatrix*}[r]
         1 &  1 &  0 \\
         0 & -1 &  1 \\
         0 &  0 & -1
    \end{pmatrix*},
\]
as one can verify by checking (for example with \texttt{GAP} \cite{GAP4}) that $\gen{x^2, y^2}$ is a subgroup of index $16$ isomorphic to the integral Heisenberg group, on which the representation is easily seen to be faithful.
We note that $S$ is torsion-free, which can be proved by writing it as the free product with amalgamation of two Klein bottle subgroups:
\begin{align*}
    & \gp{x, y}{(xy)^2 (xy^{-1})^2, \, (yx)^2 (yx^{-1})^2} \\
    &\cong \gp{x, y}{(xy)^2 (xy^{-1})^2, \, (xy)^2 (yx)^2} \\
    &\cong \gp{a, b, w, x, y}{a^2 b^2, \, x w^2 x, \, a = xy, \, b = x^{-1} y, \, w = yxy} \\
    &\cong \gp{a, b, w, x, y}{a^2 b^2, \, w^2 x^2, \, a = xy, \, b a^{-1} = x^{-2}, \, w a^{-1} = y} \\
    &\cong \gp{a, b, w, x}{a^2 b^2, \, w^2 x^2, \, a^2 = x w, \, b a^{-1} = x^{-2}} \\
    &\cong \gp{a, b}{a^2 b^2} *_{\Z^2} \gp{w, x}{w^2 x^2}
\end{align*}
This means the representation is faithful on all of the group $S$.
From the presentation we conclude that $\phi \colon S \to S \colon x \mapsto y, y \mapsto x^{-1}$ is a homomorphism and thus an order $4$ automorphism.
It is a straightforward computer verification to prove:

\begin{thm}
    The element
    \begin{multline*}
    \nu = x + x^{-1} + y + y^{-1} + x y + x^{-1} y^{-1} + y x^{-1} + y^2 + y^{-1} x + y^{-2} \\
        + x^2 y + x y^{-1} x + x y^{-2} + x^{-2} y^{-1} + x^{-1} y x^{-1} + x^{-1} y^2 + y x y + y^{-1} x^{-1} y^{-1} \\
        + x^2 y^{-1} x + x y x^2 + x^{-2} y x^{-1} + x^{-1} y^{-1} x^{-2} + y x^{-2} y^{-1} + y^{-1} x^2 y \\
        + x^2 y x^2 + x y^{-1} x^2 y + x^{-2} y^{-1} x^{-2} + x^{-1} y x^{-2} y^{-1} + x^2 y^{-1} x^2 y
    \end{multline*}
    of $\F_2[S]$ is a $\phi$-unitary unit, that is, $\nu^{-1} = \phi(\nu)^*$.
\end{thm}

Thus $\phi^2$ is a non-trivial automorphism that fixes the unit; the symmetry exhibited by $\nu$ and its inverse is order $4$ but isomorphic to $\Z/4$ rather than $\Z/2 \times \Z/2$ as was the case for $P$.

\subsection*{Acknowledgements}

This work was funded by the European Union (ERC, SATURN, 101076148) and the Deutsche Forschungsgemeinschaft (DFG, German Research Foundation) under Project-ID 506523109 (Emmy Noether) and under Germany's Excellence Strategy EXC 2044--390685587 and EXC-2047/1 -- 390685813.
The author gratefully acknowledges the granted access to the Bonna cluster hosted by the University of Bonn.

The Max Planck Institute for Mathematics in the Sciences hosted a stimulating workshop in Leipzig in April 2023 on \emph{Solving hard polynomial systems}.
I thank the organizers and the other participants, especially Georgy Scholten, for helpful discussions.
I also thank Franziska Jahnke and Daniel Windisch for many interesting discussions on model theoretic approaches to proving the existence of units in characteristic zero.

\bibliography{complex-units}{}
\bibliographystyle{alpha}

\end{document}